\documentclass[final]{article}
\usepackage{amsmath,amsthm}
  \usepackage{graphics} 
\usepackage{graphicx}  \usepackage{epstopdf}
 \usepackage[colorlinks=true]{hyperref}
\hypersetup{urlcolor=blue, citecolor=red}
\usepackage{amsfonts}
\usepackage[color]{showkeys}
\definecolor{refkey}{gray}{.75}
\usepackage{xcolor}
\definecolor{rev}{RGB}{0,0,0}
\definecolor{tbp}{RGB}{255,0,0}

\newtheorem{theorem}{Theorem}[section]

\newtheorem{lemma}[theorem]{Lemma}

\def\p{\partial}
\def\tilde{\widetilde}

\def\bs{\boldsymbol}

\def\xb{\bs{x}}
\def\xib{\bs{\xi}}
\def\ab{\bs{a}}
\def\bb{\bs{b}}
\def\nb{\mathbf{n}}
\def\ub{\bs{u}}

\def\kb{\bs{\mathrm{k}}}

\def\M{\mathcal{M}}

\def\l2m{L^2(\M)}
\def\w1pm{W^{1,p}(\M)}
\def\w1pO{W^{1,p}(\Omega)}
\def\h10m{H^1_0(\M)}
\def\Gper{G_{\mathrm{per}}^{\xb,z}}
\def\phi{\varphi}

\def\C01{C^{0,1}\left(\overline{\Omega}_T\right)}

\ifpdf
  \DeclareGraphicsExtensions{.eps,.pdf,.png,.jpg}
\else
  \DeclareGraphicsExtensions{.eps}
\fi

\begin{document}
\title{A global well-posedness result for the three-dimensional inviscid
  quasi-geostrophic equation over a cylindrical domain}

\author{Qingshan Chen}




\maketitle

\begin{abstract}
The three-dimensional quasi-geostrophic equation is considered over a
cylindrical domain with a multiply connected horizontal
cross-section. Homogeneous Neumann boundary conditions, tantamount to
homogeneous density fields, are imposed on the top and bottom
surfaces, while no-flux boundary conditions combined with constant
circulations are imposed on the lateral boundary loops. The global existence and
uniqueness of a generalized solution is proven, provided that
the initial potential vorticity (PV) field is essentially bounded. If the
initial PV field is differentiable, then the solution is shown to
satisfy the system in the classical sense. 
\end{abstract}

\section{Introduction}\label{s:intro}
On the hierarchy of fluid models for large-scale geophysical
flows, the quasi-geostrophic equations occupy an important
place as they are much simpler than other more
comprehensive ones, such as the primitive equations, and yet,
 manage to capture one of the most important dynamics of
geophysical flows, namely the transport of the potential vorticity by
its self-generated velocity field. Due to its
simple and yet relevant nature, the (barotropic) quasi-geostrophic was
selected by Charney et (\cite{Charney1950-sx}) as one of the few
problems to be 
tackled by the very first electronic computer, ENIAC, some seventy-five
years ago. Ever since then, the QG
equations have remained an important model in the study of geophysical flows, see
eg.~\cite{Holland1978-us,McWilliams1989-ui,Pedlosky1964-wb,Holland1978-us},
among many others.

The QG 
equations typically take the form of a single scalar transport equation (see
next section for the form). Behind this simple form, however, lie
several different configurations, depending the assumed density
profiles. The simplest 
case involves a uniform density 
profile, and the resulting equation is called the barotropic QG equation. If a
uniform density profile is considered as too 
simplistic for the problem at hand, then a multi-layer QG model can be
used, which assumes that fluid is stratified into multiple layers of
distinct densities, with lighter ones flowing over heavier
ones. Finally, in the most realistic setting, a continuum 
of density profile is assumed, and the three-dimension QG equation
results. The current article deals with this last
setting.

Due to their important roles in the
study of geophysical flows, as well as their close connections to
other fluid models (eg. Euler equations, SQG), the well-posedness of the QG
equations has long attracted mathematicians' attention. Dutton
(\cite{Dutton1974-oo}) considers the three-dimensional QG equation 
in a rectangular box, with periodic
boundary conditions on the sides, and homogeneous Neumann boundary
conditions on the top and bottom. The global
existence of a generalized solution and the uniqueness of a classical
solution are established. Bourgeois \& Beale (\cite{Bourgeois1994-tv})
study the model under a similar setting, and the existence of a global
strong solution is proven. To this setup,
Desjardins \& Grenier (\cite{Desjardins1998-hb}) add the transport
equations of the 
density (vertical gradient of the streamfunction) with a diffusive
Ekman pumping effect on the top and bottom boundary surfaces. The
existence of a global weak solution is proven. Puel \& Vasseur
(\cite{Puel2015-mw}) consider the 3D QG on the upper 
half space, with the density transport equation, but without the Ekman
pumping effect, posed at the bottom boundary. The global existence of
a weak solution is proven. With an added Ekman pumping effect, Novack \& Vasseur
(\cite{Novack2018-op}) prove the existence \& uniqueness of a global strong solution to
the system. With initial data from Lebesgue
spaces, Novack (\cite{Novack2019-ch}) establishes the global existence of a weak
solution. Novack (\cite{Novack2020-fl}) further confirms the
non-uniqueness of weak solutions in the $C^\zeta_{tx}$, $\zeta <
\frac{1}{5}$, class. Novack \& Vasseur (\cite{Novack2020-gi, Novack2020-vv}) consider the 3D QG on
a cylindrical domain, with non-diffusive transport equations on the top \&
bottom boundaries. The global existence of a
weak solution \& the local existence of a
classical solution are proven.

Among the aforementioned works, especially in the non-diffusive 
cases, only partial well-posedness results,
either existence or uniqueness, have been
proven, in contrast with the
case of the 2D Euler equation, where both the
existence and uniqueness of weak and classical solutions have been
established (\cite{Yudovich1963-bj,Kato1967-pe}). It is
our goal to pursue the global well-posedness of the QG equations in a
non-viscous setting with realistic boundary conditions. Chen
(\cite{Chen2019-fh}) studies the 2D QG equation under a free
surface over a bounded simply-connected domain with no-flux conditions on the
boundary. The global existence \& uniqueness of a
weak solution is proven. The same result is
then extended in \cite{Chen2019-bn} to the multilayer QG equations,
again over a bounded and 
simply-connected domain.

The current work aims to study the global well-posedness of
the 3D QG equation. To make the problem more
tractable, we first replace the transport equations from the top and bottom
boundaries (see \cite{Desjardins1998-hb, Novack2020-gi}) with 
homogeneous Neumann boundary conditions, without any Ekman pumping
effect, of course. The homogeneous Neumann boundary conditions
physically amount to 
requiring the density fields to be homogeneous on the top and bottom
surfaces, a condition that largely holds true over regions where the
quasi-geostrophic equations are applicable
(\cite{Cushman-Roisin2011-en}).  Now, with the top and bottom boundaries
taken care of, 
the setup of the horizontal boundary conditions/constraints remains as
the main obstacle. The difficulty stems from the
imbalance between the dimension of the spatial domain (3D) and the
dimension of the flow dynamics (fundamentally 2D). We consider the
flow in a cylindrical domain with a 
multiply connected cross-section. Drawing on
mathematical derivations of Grenier, Novack \& Vasseur (see references
given above), and physicists' intuition
(\cite{Pedlosky1987-gk,Holm1986-rb,Liu1996-su}), we 
propose a setup that combines the no-flux boundary conditions with the
constant circulation constraints (\cite{ChenUnknown-gp}). 
This setup 
allows us to explicitly specify the equations and conditions for the
Green function for this BVP, and establish the needed regularity of the
streamfunction. With this regularity on the streamfunction and the
velocity field derived from it, it turns out that the 3D QG indeed
behaves more like a two-dimensional model than a 3D one, and the
classical techniques (see references from Yudovich and Kato, and also
\cite{Marchioro1994-yt})  for 2D flows  can be applied. The global existence
and uniqueness is proven, when the potential vorticity is essentially
bounded; when the initial PV is in $C^1$, the
solution is then shown to be a classical solution of the system.
To the
best of our knowledge, these results are new for the inviscid 3D QG
equation.

The
remainder of the article is organized as 
follows. In section 2, a complete setup of the
model is given. In section 3, the well-posedness
of this elliptic boundary value problem as well as the regularity of
the streamfunction are treated. The global existence and uniqueness of
a weak solution are established in Section 4. Higher regularity is
obtained with smoother initial 
data in Section 5. Finally, a summary is given in Section 6.

\section{Specification of the physical model}
\begin{figure}[h]
  \centering
  \includegraphics[width=4in]{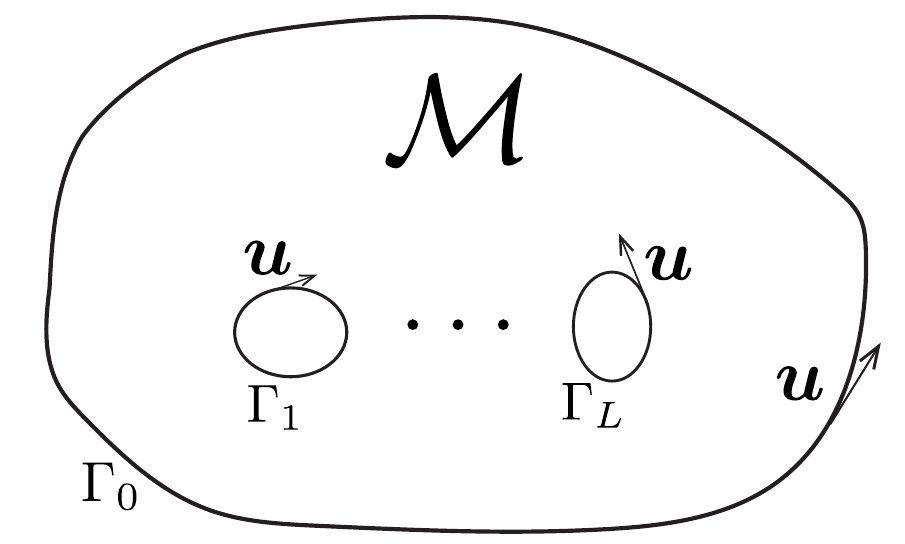}
  \caption{A multiply-connected horizontal cross-section of the
    cylindrical domain.}
  \label{fig:domain}
\end{figure}
We consider the inviscid three-dimensional quasi-geostrophic (QG)
equations on a cylindrical domain $\Omega = \M\times(0,\,1)$, where
$\M$ is a general multiply connected two-dimensional domain with
boundaries $\p\M = \cup_{l=0}^L \Gamma_l$, with $\Gamma_0$ designating
the exterior boundary loop (Figure \ref{fig:domain}).
The main equation takes the elegant form of a scalar transport
equation,
\begin{equation}
  \label{eq:11}
\dfrac{\p}{\p t} q + \ub\cdot\nabla q = 0.  
\end{equation}
Here $q \equiv q(\xb,z,t)$, with $x\in\M$ and $z\in (0,\,1)$, represents the
potential vorticity in a three-dimensional space, $\ub \equiv
\left(u(\xb,z,t),\,v(\xb,z,t)\right)$ the horizontal velocity field in a
three-dimensional space, and $\nabla \equiv (\p_x,\,\p_y)$ the
horizontal gradient operator. The potential vorticity $q$ and the
velocity field $\ub$ are connected through a third variable, the
streamfunction $\psi \equiv \psi(x,\,z,t)$, as follows,
\begin{align}
  &q = \Delta\psi + \dfrac{\p}{\p z}\left(\dfrac{1}{S}\dfrac{\p
  \psi}{\p z}\right),\label{eq:12}\\
  &\ub = \nabla^\perp\psi.\label{eq:13}
\end{align}
In the above, $\nabla^\perp \equiv \kb\times\nabla \equiv
(-\p_y,\,\p_x)$ stands for the horizontal skew gradient operator, and
$S\equiv S(z)$ the strictly positive stratification factor, i.e.~$S(z)
\ge \delta$ for some $\delta>0$  and for 
$\forall\,z\in(0,\,1)$. To simplify the presentation below, we take
$S\equiv 1$ to be a constant; the general case of a variable
stratification factor can be handled through a rescaling in the
vertical direction.  


In order to recover the streamfunction $\psi$ from the PV $q$ using
the elliptic equation \eqref{eq:12}, various boundary conditions are
needed. 
No-flux boundary conditions are applied on the lateral boundary
$\p\M\times(0,1)$, which means that, on each boundary loop, the streamfunction must take
constant values at any vertical level, i.e.
\begin{equation}
    \label{eq:9}
  \psi(\xb,\,z,\,t)|_{\Gamma_l} =
  \textrm{independent of }\xb,\qquad 0\leq l\leq L.
\end{equation}

Boundary conditions on the top and bottom surfaces are
supplied by the physical constraints that density should be simply
transported around, without change, by the flow. Under the hydrostatic
approximation, density is proportional to the vertical pressure
gradient. Therefore, on the top and bottom surface, a physical
boundary condition is given by
\begin{equation*}
  \dfrac{\p}{\p t}\dfrac{\p\psi}{\p z} +
  \ub\cdot\nabla\dfrac{\p\psi}{\p z} = 0, \qquad \forall\,\xb\in\M,\,\,z
  = 0,\,1.
\end{equation*}
This transport equation is the physical model that inspires  the
well-known surface QG equation, under which  the vertical depth has
been extended to infinity, and thus eliminated, and the vertical
derivative has been generalized to the fractional Laplace operator
$\Delta^{\frac{1}{2}}$ (REF). This physical boundary condition
is relevant, interesting, and also challenging. For this study, we
consider a simplified version of the condition, where the density
remains uniform on the top and bottom surfaces, i.e.
\begin{equation*}
  \dfrac{\p\psi}{\p z} = 0,\qquad \forall\,(\xb,\,z)\in\M\times\{
  0,\,1\}. 
\end{equation*}

The system is under-determined due to the fact that the lateral
boundary values of the streamfunction $\psi$ remain unknown. To fully
determine the 
system, more constraints is needed. With the streamfunction
representing the pressure or its fluctuations, fixing the values of
the streamfunction along the boundary loop through the classical
Dirichlet boundary conditions will not be physical. On the other hand,
it is well known (Pedlosky, Arnold, Yudovich) that, for
two-dimensional or predominantly two-dimensional flows subject to
conservative forces, the
circulation along each closed boundary loop is conserved in
time. Thus, in stead of the boundary values of $\psi$, one can
specify the circulations of the velocity field along each
boundary loop,
\begin{equation}
  \label{eq:1}
  \int_{\Gamma_l} \ub\cdot{\bs\tau} ds \equiv \int_{\Gamma_l}
  \dfrac{\p\psi}{\p{\bs n}} ds = c_l(z), \qquad 0\leq l\leq L, 
\end{equation}
where $c_l(z)$ is a given function that depends on the vertical
coordinate $z$, but independent of $\xb$ or $t$.  

So far, it is obvious that any solution to the elliptic system can not
be unique because any solution plus a constant will still be a
solution. To eliminate this arbitrariness from the system, one more
constraint is needed, and a natural one is given by
\begin{equation}\label{eq:42}
  \int_0^1 \int_\M \psi(\xb,\,z,\,t)d\xb dz = 0.
\end{equation}

In summary, at every instant $t$ in time,
when the potential vorticity $q$ is known, the streamfunction $\psi$ can be
found by solving a non-standard three-dimensional elliptic boundary
value problem,
\begin{equation}
  \label{eq:43a}
\left\{
  \begin{aligned}
\Delta\psi + \frac{\p^2 \psi}{\p z^2}&=
   q(\xb,\,z,\,t), & &\xb\in\M,\,z\in(0, 1),\\
   \psi(\xb,\,z,\,t )|_{\Gamma_l} &= \textrm{indep.~of~}\xb, & &
   \forall\, 0\leq l\leq L,\,\,z\in(0,1),\\
   \dfrac{\p\psi}{\p z} &= 0, &  &\xb\in\M,\,z = 0,\,1,\\
   \int_{\Gamma_l}
   \dfrac{\p\psi}{\p{\bs n}} ds &= c_l(z), &  &0\leq l\leq L, \\
\int_0^1\!\int_\M \psi(\xb,z,t) d\xb dz &= 0, & & t>0.
  \end{aligned}\right.
\end{equation}

Integrating the first equation of \eqref{eq:43a} over $\Omega$ and
utilizing the boundary conditions specified, we find that
\begin{equation}
  \label{eq:43}
  \sum_0^L \int_0^1 c_l(z)  dz = \int_0^1\!\int_\M q(\xb,z)d\xb dz.
\end{equation}
This is a compatibility condition that the PV field and the boundary
data have to satisfy.

The elliptic system \eqref{eq:43a} plays an important role in the 3D QG system, as it
connects the transport velocity field, $\ub$, with the transportee, the
potential vorticity $q$, via the streamfunction $\psi$. In numerical
simulations with the QG equations, this system will have to be solved
at every time step
in order to recover the velocity from the potential vorticity (see
\cite{Chen2018-bu} for the barotropic case).

The potential vorticity can be split into a
volume averaged part and a transient part,
\begin{equation*}
  q(\xb,\,z,\,t) = q^0 + q^1(\xb,\,z,\, t),
\end{equation*}
where
\begin{equation*}
  q^0(z) = \dfrac{1}{|\Omega|}\int_0^1 \int_\M q(\xb,\,z,\,t) d\xb dz
\end{equation*}
is the volume averaged part, and the transient part $q^1$ is then
volume average zero by design. 
It is important to note that this averaged quantity $q^0$ is time independent,
because of the conservation of the potential vorticity at each $z$
level, thanks to the transport equation \eqref{eq:11} and the no-flux boundary condition
\eqref{eq:9}.

The elliptic system \eqref{eq:43a} has a corresponding two-way splitting,
\begin{equation*}
  \psi = \psi^1 + \psi^0,
\end{equation*}
with $\psi^1$ solving
\begin{equation}
  \label{eq:a44}
\left\{
  \begin{aligned}
\Delta\psi^1 + \frac{\p^2 \psi^1}{\p z^2}&=
   q^1(\xb,\,z,\,t), & &\xb\in\M,\,z\in(0, 1),\\
   \psi^1(\xb,\,z,\,t )|_{\Gamma_l} &= \textrm{indep.~of~}\xb, & &
   \forall\, 0\leq l\leq L,\,\,z\in(0,1),\\
   \dfrac{\p\psi^1}{\p z} &= 0, &  &\xb\in\M,\,z = 0,\,1,\\
   \int_{\Gamma_l}
   \dfrac{\p\psi^1}{\p{\bs n}} ds &= 0, &  &0\leq l\leq L, \\
  \int_0^1 \int_\M \psi^1(\xb,z,t) d\xb dz&= 0, & &
  \end{aligned}\right.
\end{equation}
and $\psi^0$ solving
\begin{equation}
  \label{eq:a45}
\left\{
  \begin{aligned}
\Delta\psi^0 + \frac{\p^2\psi^0}{\p z^2} &=
   q^0, & &\xb\in\M,\,z\in(0, 1),\\
   \psi^0(\xb,\,z,\,t )|_{\Gamma_l} &= \textrm{indep.~of~}\xb, & &
   \forall\, 0\leq l\leq L,\,\,z\in(0,1),\\
   \dfrac{\p\psi^0}{\p z} &= 0, &  &\xb\in\M,\,z = 0,\,1,\\
   \int_{\Gamma_l}
   \dfrac{\p\psi^0}{\p{\bs n}} ds &= c_l(z), &  &0\leq l\leq L, \\
  \int_0^1 \int_\M \psi^0(\xb,z,t) d\xb &= 0. & &
  \end{aligned}\right.
\end{equation}
Among the two components of $\psi$, 
$\psi^0$ is time independent, since both $q^0$ and $c_l's$
are, and $\psi^1$ is the only time dependent component of $\psi$.
We let
\begin{equation*}
  \ub^1 = \nabla^\perp\psi^1, \qquad \ub^0 =
  \nabla^\perp \psi^0. 
\end{equation*}
Thus $\ub^1$ is the time dependent (transient) part of $\ub$, and
$\ub^0$ the time-independent part. Replacing $q$ by $q^0 + 
q^1$ and $\ub$ by $\ub^0 + \ub^1$ in \eqref{eq:11}, we have
\begin{equation*}
\dfrac{\p}{\p t} q^1 + \ub^1\cdot\nabla q^1 +
\ub^0\cdot\nabla q^1 = 0, \qquad\M\times(0,\,1).
\end{equation*}
It is clear that the nonlinearity of the 3D QG equations comes from
the coupling of $\ub^1$ and $q^1$; the steady-state velocity
filed $\ub^0$ is independent of $q^1$, and provides a constant
drifting velocity to the system. The nonlinearity is of course the
main source of difficulty for this system, and the focus of the
current paper. Hence, in the rest of the
paper, we assume that potential vorticity is horizontally average
zero initially,
\begin{equation}
  \label{eq:3}
  \int_\M q(\xb,z,t)d\xb = 0,\qquad \forall\, z\in (0,\,1),\,\, t>0.
\end{equation}
and that the circulations along all the boundary
loops are also zero ($c_l = 0$, for $0\leq l\leq L$) as well.

In summary, the elliptic
boundary value problem \eqref{eq:43a} now has the form
\begin{equation}
    \label{eq:2}
\left\{
  \begin{aligned}
\Delta\psi + \frac{\p^2 \psi}{\p z^2} &=
   q(\xb,\,z,\,t), & &\xb\in\M,\,z\in(0, 1),\\
   \psi(\xb,\,z,\,t )|_{\Gamma_l} &= \textrm{indep.~of~}\xb, & &
   \forall\, 0\leq l\leq L,\,\,z\in(0,1),\\
   \dfrac{\p\psi}{\p z} &= 0, &  &\xb\in\M,\,z = 0,\,1,\\
   \int_{\Gamma_l}
   \dfrac{\p\psi}{\p{\bs n}} ds &= 0, &  &0\leq l\leq L, \\
  \int_0^1\!\int_\M \psi(\xb,z,t) d\xb d z &= 0. & &,
  \end{aligned}\right.
\end{equation}
with the PV $q$ satisfying the constraint \eqref{eq:3}.

The initial state of the system is given by
\begin{equation}
  \label{eq:14}
  q(\xb,\,z,\,0) = q_0(\xb,\,z),
\end{equation}
where $q_0(\xb,z)$ is an arbitrary function satisfying the condition
\eqref{eq:3}.

The 3D QG system to be studied comprises of the equations
\eqref{eq:11}, \eqref{eq:13}, the elliptic system \eqref{eq:2}, and
the initial condition \eqref{eq:14}.
The analysis presented here can be
easily adapted for the case of nonzero average PV and/or none zero
circulations along the boundary loops. 
\section{Well-posedness and solution regularity of the non-standard
  three-dimensional elliptic boundary value problem}\label{s:elliptic}

In this section, the dependence on time $t$ is temporarily
suppressed.\\

\noindent{\bf Existence and uniqueness of a weak ($H^1$) solution}\\
One assumes that $\psi$ is a smooth solution to the elliptic BVP
\eqref{eq:2}, and $\varphi$ another smooth function satisfying the
same boundary conditions. Multiplying the elliptic equation
$\eqref{eq:2}_1$ by $\varphi$ and integrating by parts over $\Omega$,
and using the zero circulation conditions on the boundary
loops and the homogeneous Neumann boundary conditions on the top and
bottom surfaces, we have
\begin{equation}
  \label{eq:4}
  \int_0^1\!\int_\M \left(\nabla\psi \cdot\nabla\varphi +
    \dfrac{\p\psi}{\p z} \dfrac{\p\varphi}{\p z}\right) d\xb
  dz= -\int_0^1 \int_\M q\varphi d\xb dz.
\end{equation}
Defining
\begin{align*}
  B(\psi,\,\varphi) &\equiv   \int_0^1\!\int_\M \left(\nabla\psi
    \cdot\nabla\varphi + \dfrac{\p\psi}{\p z}
    \dfrac{\p\varphi}{\p z}\right) d\xb dz,\\
  l_q (\varphi) &\equiv  -\int_0^1 \int_\M q\varphi d\xb dz,
\end{align*}
one can then write the equation above as
\begin{equation*}
  B(\psi,\,\varphi) = l_q(\varphi). 
\end{equation*}

One defines
\begin{equation}
  \label{eq:5}
  V = \left\{ \psi \in H^1(\Omega)\,\,|\,\, \psi|_{\Gamma_l} =
    \textrm{indep.~of }\xb,\,\,0\leq l \leq L,\textrm{ and } \int_0^1\!\int_\M
    \psi(\xb,z) d\xb dz = 0\right\}, 
\end{equation}
which is equipped with norm
\begin{equation*}
  \| \psi \| = \left(\int_0^1 \int_\M \left(|\nabla\psi|^2 +
      \left| \dfrac{\p\psi}{\p z}\right|^2 \right) d\xb
    dz\right)^{\frac{1}{2}}. 
\end{equation*}
Given  the zero-average condition on
$\psi$, this norm makes $V$ a complete Hilbert space. One also defines
\begin{equation*}
  H = \left\{ q\in L^2(\Omega)\,\,\left|\,\,\int_0^1\!\int_\M q(\xb,z) d\xb
      dz \right. = 0\right\}. 
\end{equation*}
This space is a Hilbert space under the usual $L^2$ norm. 
A weak formulation of the elliptic boundary value problem \eqref{eq:2}
can then be stated as follows:
\begin{quote}{\it
  For a given $q\in H$, find $\psi\in V$ such that
  \begin{equation}
      \label{eq:6}
    B(\psi,\,\varphi) = l_q (\varphi),\qquad \forall\,\,\varphi \in
    V. 
  \end{equation}}
\end{quote}

One notes that the bilinear form $B(\cdot,\,\cdot)$ is bounded and
coercive on the space $V$, and $l_q$, for any given $q\in H$,  is a
bounded linear functional on $V$. By the classical Lax-Milgram
Theorem, the weak problem has a unique solution in $V$.

In order to recover the PDE and the boundary conditions from
the weak formulation, one integrates by parts on the left-hand side of
\eqref{eq:4}, and obtains
\begin{multline*}
  -\int_0^1 \int_\M \left(\Delta \psi + \dfrac{\p}{\p z}\left(
       \dfrac{\p \psi}{\p z} \right) - q\right) \varphi d\xb
  dz + \sum_{l=0}^L \int_0^1 \left(\int_{\Gamma_l}
  \dfrac{\p\psi}{\p\nb} ds\right) \left.\varphi\right|_{\Gamma_l}  dz \\ {}+
\int_\M 
  \left.\dfrac{\p\psi}{\p
      z}\right|_{z=1} \varphi(x,1) d\xb- \int_\M
  \left.\dfrac{\p\psi}{\p
      z}\right|_{z=0} \varphi(x,0) d\xb = 0.
\end{multline*}
One notes that the equation must hold for arbitrary test function
$\varphi$ from $V$, and this implies that each individual
integral from the above must vanish. we have the following
\begin{theorem}
  For every $q\in H$, the weak formulation \eqref{eq:6} has a unique solution $\psi$ in
  $V$, and the 
  solution satisfies the elliptic PDE in the distribution sense,
  the homogeneous Neumann boundary conditions at the top and bottom in
  the sense that
  \begin{equation*}
\int_\M
  \left.\dfrac{\p\psi}{\p
      z}\right|_{z} \varphi(\xb,z) d\xb = 0,\qquad \textrm{for } \forall\,\varphi \in V,\,\, z =
  0\textrm{ or } 1,
\end{equation*}
and the circulation restrictions along each boundary loop $\Gamma_l$,
$0\leq l\leq L$,  in the sense
that
\begin{equation*}
\int_0^1 \left(\int_{\Gamma_l}
  \dfrac{\p\psi}{\p\nb} ds\right) \varphi|_{\Gamma_l}   dz = 0,\qquad \forall \,\varphi \in V.   
\end{equation*}
\end{theorem}

\vspace{3mm}
\noindent{\bf Solution by Green's function}\\
Construction and specification of the Green's function
$G^{\xb,z}(\xib,\eta)$ for the elliptic boundary value problem
\eqref{eq:2} are given in Appendix \ref{sec:green}. 
The solution to system \eqref{eq:2} can be expressed, using this
Green's function,
\begin{equation}
  \label{eq:17}
  \psi(\xb,z) = \int_\Omega G^{\xb,z}(\xib, \eta) q(\xib,\eta) d \xib d \eta. 
\end{equation}

In the case \eqref{eq:43a} of non-zero-average PV and  non-homogeneous
boundary conditions, the 
formula \eqref{eq:17}  needs to be expanded to include the boundary
integrals, and becomes
\begin{equation}
    \label{eq:18}
    \psi(\xb,z) = \int_\Omega G^{\xb,z}(\xib, \eta) q(\xib,\eta) d \xib d \eta - 
    \sum_{l=0}^L \int_0^1 \left.G^{\xb,z}(\xib,\eta)\right|_{\Gamma_l}  c_l(\eta) d\eta. 
\end{equation}

\noindent{\bf Quasi-Lipschitz continuity of the velocity field}\\
In the following, we focus on the regularity of the solution, in particular, the
quasi-Lipschitz continuity of its gradient field field.
This kind of regularity is typically established using the
the expression \eqref{eq:17} and estimates on the Green's function. 
Compared with the standard elliptic boundary value problem with, e.g.,
homogeneous Dirichlet boundary conditions
(\cite{Gilbarg1983-pq,Gr-uter1982-fy}), our 
non-standard BVP \eqref{eq:2} presents an extra layer of difficulty
associated with the sharp edges between top/bottom surfaces and the lateral
boundary. This difficulty, however, is essentially circumvented in the
construction of the Green's function $G^{\xb,z}(\xib,\eta)$, where the
boundary problem and domain are expanded over the $z$-axis, with
periodic boundary conditions in this direction.  More details can be found in
Appendix \ref{sec:green}.  

The regularity results on the streamfunction, needed later, are
summarized as follows. 
\begin{theorem}\label{lem:elliptic-reg}
  Assume that the lateral side of the cylinder is smooth with $\p\M$
  being at least piecewise $C^1$, and that $q\in L^\infty(\Omega)$. Let $\psi$
  be the unique solution to the boundary value problem
  \eqref{eq:2}. Then
  the first derivatives of $\psi$ are
    quasi-Lipschitz continuous in the sense that
    \begin{align}
     \left| \nabla_3 \psi(\xb_1,z_1) - \nabla_3 \psi(\xb_2,z_2)
   \right| &\leq C(\Omega, |q|_\infty) \lambda(d),  \label{eq:214j}
    \end{align}
where $d = \sqrt{|\xb_1- \xb_2|^2 + (z_1 - z_2)^2}$, and
  \begin{equation*}
    \lambda(d) = \left\{
      \begin{aligned}
        &(1-\log d) d,& 0< d < 1&,\\
        &1, & d \ge 1&.
      \end{aligned}\right.
  \end{equation*}
 \end{theorem}

With the edge issue taken care of,  the proof of this regularity
result  is
similar to the two-dimensional case 
(\cite{Yudovich1963-bj,Kato1967-cj}), with only minor changes
needed. Details are given in Appendix \ref{sec:lipschitz}.





When the boundary conditions are non-homogeneous, then the boundary
data will have to satisfy certain compatibility conditions along the
sharp edge to ensure
higher regularity of the solution. In our case, if the streamfunction
satisfies the homogeneous Neumann boundary conditions on the top and
bottom, but non-homogeneous constant circulation condition along each
horizontal boundary loops (\eqref{eq:43a}), then the following
compatibility conditions will have to be satisfied to ensure higher
regularity of the solution,
\begin{equation}
  \label{eq:8}
  \dfrac{\p}{\p z} c_l(z) = 0, \qquad z = 0,\,1,\,\, 0\leq l\leq L. 
\end{equation}
See \cite{Novack2020-gi} for a case of more complex boundary
conditions on the top and bottom. In our case, the circulations are
assumed to be homogeneous, and thus the compatibility conditions are
satisfied automatically.

\section{Existence and uniqueness of a weak solution to the 3D
  QG}\label{s:exist} 
The center piece of this analysis is the so-called
flowmap, denoted as
\begin{equation*}
  \Phi_{z,t}:\,\M\longrightarrow \M. 
\end{equation*}
The parameters $z\in(0,\,1)$ and $t\in\mathbb{R}$ refer to the
vertical level and time, respectively. The flow map at each $z$-level
is defined by the incompressible horizontal velocity field $\ub$: for
each $\ab\in\M$, $\Phi_{z,t}(\ab)$ solves the initial value problem
\begin{equation}
  \label{eq:20}
  \left\{
    \begin{aligned}
      \dfrac{\p\Phi_{z,t}(\ab)}{\p t} &= \ub(\Phi_{z,t}(\ab),\,z,\,t),\\
        \Phi_{z,0}(\ab) &= \ab. 
    \end{aligned}\right.
\end{equation}
Equivalently, the IVP can also be expressed as
\begin{equation}
  \label{eq:21}
  \Phi_{z,t}(\ab) = \ab + \int_0^t 
  \ub(\Phi_{z,\tau}(\ab),\,z,\,\tau)d\tau,\qquad\forall\,\ab\in\M. 
\end{equation}

With the flow map, the potential vorticity can be expressed using its
initial state,
\begin{equation}
  \label{eq:26}
  q(\xb,\,z,\,t) = q_0(\Phi_{z,-t}(\xb),\,z),
\end{equation}
where $\Phi_{z,-t}$ can be interpreted as the inverse of
$\Phi_{z,t}$. It is easily verified that 
the PV defined this way satisfies the QG equation \eqref{eq:11}, for
\begin{align*}
  0 &= \dfrac{\p }{\p t} q_0(\ab,\,z,\,t)\\
    &= \dfrac{\p }{\p t} q(\Phi_{z,t}(\ab),\,z,\,t) \\
  &= \dfrac{\p}{\p t} q + \ub\cdot\nabla q.
\end{align*}

Thus, if $(q,\,\ub)$ is a smooth solution of the QG system
\eqref{eq:11}, \eqref{eq:13}, \eqref{eq:2}, and \eqref{eq:14}, then
$(q,\,\ub,\,\Phi)$ satisfies the equations \eqref{eq:13},
\eqref{eq:2}, \eqref{eq:20}, and \eqref{eq:26}
simultaneously. Conversely, if $(q,\,\ub,\,\Phi)$ is a smooth
solution of \eqref{eq:13},
\eqref{eq:2}, \eqref{eq:20}, and \eqref{eq:26}, then
$(q,\,\ub)$ solves the original QG system as well. However,
there is no guarantee that the PV defined using the flow map in
\eqref{eq:26} is differentiable in space and/or time. Hence the
system of equations \eqref{eq:13},
\eqref{eq:2}, \eqref{eq:20}, and \eqref{eq:26} generalizes the
original QG system of \eqref{eq:11}, \eqref{eq:13},  \eqref{eq:2}, and
\eqref{eq:14}, and the solution of the former is a weak solution
of the latter.

In the following, we study the well-posedness of the generalized QG
system. The main result is summarized in
\begin{theorem}\label{t:1}
  Let $q_0(\xb,\,z)\in L^\infty(\Omega)$ satisfying
  \eqref{eq:3}. There exists a unique 
  solution $(q,\,\ub,\,\Phi)$ to the generalized QG system \eqref{eq:13},
\eqref{eq:2}, \eqref{eq:20}, and \eqref{eq:26}. 
\end{theorem}

We follow Kato (\cite{Kato1967-cj}) and Marchioro \& Pulvirenti
(\cite{Marchioro1994-yt}), and prove this result using an iterative
procedure. We let
\begin{equation*}
  q^0(\xb,\,z,\,t) =
  q_0(\xb,\,z),\qquad\forall\,(\xb,\,z)\in\Omega,\,t\ge 0,
\end{equation*}
where $q_0$ is the given initial state of the system. Once $q^n$ is
known, we find its next iteration, $q^{n+1}$, via the following\\

\noindent{\bf Iterative Procedure}
\begin{enumerate}
\item Solve the elliptic system \eqref{eq:2} for $\psi^n$, with $q$
  taken to be $q^n$.
\item Set $\ub^n = \nabla^\perp \psi^n$.
\item Solve the initial value problem \eqref{eq:20} for
  $\Phi^n_{z,t}$, with $\ub$ taken to be $\ub^n$.
\item Update the PV by
  \begin{equation}
    \label{eq:44}
  q^{n+1}(\xb,\,z,\,t) = q_0(\Phi^n_{z,-t}(\xb),\,z). 
  \end{equation}
\end{enumerate}

Assuming that each of these steps is well defined, repeating this process
generates a sequence of triplets of functions,
$\{(q^n,\,\ub^n,\,\Phi^n_{z,t})\}_{n\ge 0}$. The well-posedness of the
system is established once we show that the sequence converges to a
unique triplet $(q,\, \ub,\, \Phi_{z,t})$ that solves the QG system.

The well-posedness of the elliptic system \eqref{eq:2} has been
established in Section \ref{s:elliptic}, together with the
quasi-Lipschitz continuity of the gradient field. Thus Steps 1 \& 2 in
the above are well defined, provided that the initial data are
sufficiently smooth. Regarding Step 3 and the initial value problem
\eqref{eq:20} or \eqref{eq:21},  Kato and Marchioro \& Pulvirent pointed
out in the 2D Euler case that the quasi-Lipschitz continuity of the gradient
(velocity) field is sufficient to ensure the global well-posedness of
the initial value problem defining the flow map. It turns out that the
same is also true for the case of the 3D QG equation, because the
velocity field of the 3D system remains two-dimensional (horizontal),
and is defined by the horizontal skew-gradient of the
streamfunction. Specifically, we have
\begin{lemma}\label{l:1}
  Under the condition that $\ub(\xb,\,z,\,t)$ is continuous in $z$ and
  $t$, and quasi-Lipschitz continuous in $\xb$, the initial value
  problem \eqref{eq:20} has a unique solution $\Phi_{z,t}(\ab)$ for all
  $\ab\in\M$, $z\in(0,\,1)$, and $t>0$. 
\end{lemma}

\begin{proof}[Proof of Lemma \ref{l:1}]
 We establish the existence and uniqueness of a solution to this
 equation by the classical Piccard iterations. We let $T>0$ to be
 determined later, and for
 $t\in[0,\,T]$, we let
 \begin{equation*}
   \Phi_{z,t}^0(\ab) = \ab,\qquad \forall\,(\ab,\,z)\in\Omega. 
 \end{equation*}
Assuming that $\Phi^k_{z,t}$ is known, we construct its next iteration
over the same time interval $[0,\,T]$ by
\begin{equation}
  \label{eq:27}
  \Phi^{k+1}_{z,t}(\ab) = \ab +
  \int_0^t\ub(\Phi^k_{z,\tau}(\ab),\,z,\, \tau) d\tau. 
\end{equation}
We now show that $\{\Phi^k_{z,t}(\ab)\}_{k\ge 0}$ is a uniformly
converging sequence for $\forall\,(\ab,\,z,\,t)\in\Omega\times[0,\,T]$. From
\eqref{eq:27}, we have
\begin{equation*}
  \Phi^{k+1}_{z,t}(\ab) - \Phi^k_{z,t}(\ab) = \int_0^t\left(
    \ub\left(\Phi^k_{z,\tau}(\ab),\,z,\, \tau\right) -
        \ub\left(\Phi^{k-1}_{z,\tau}(\ab),\,z,\, \tau\right) \right)
      d\tau. 
\end{equation*}
Taking absolute value on both sides and using the quasi-Lipschitz
continuity of $\ub$, we obtain, for any $(\ab,\,z)\in\Omega$,
\begin{equation}\label{eq:37}
  \left|\Phi^{k+1}_{z,t}(\ab) - \Phi^k_{z,t}(\ab) \right| \leq C
  \int_0^t \lambda \left(
    \Phi^k_{z,\tau}(\ab) -\Phi^{k-1}_{z,\tau}(\ab)\right))  d\tau, 
\end{equation}
where the constant $C$ is the same as the constant in \eqref{eq:214j},
and is independent of $(\ab,\,z)$.
We note that, for any parameter $0<\epsilon<1$, the function
$\lambda(d)$ is majorized by a linear function,
\begin{equation}
  \label{eq:28}
  \lambda(d) \leq -\log\epsilon\cdot d + \epsilon, \qquad\forall\,
  d>0. 
\end{equation}
Using this inequality in the estimate above, we have
  \begin{equation}
    \label{eq:29}
  \left|\Phi^{k+1}_{z,t}(\ab) - \Phi^k_{z,t}(\ab) \right| \leq
  -C \log\epsilon 
  \int_0^t \left|
    \Phi^k_{z,\tau}(\ab) -\Phi^{k-1}_{z,\tau}(\ab)\right|  d\tau +
  C\epsilon T. 
    \end{equation}
We set $U = \sup_{(\xb,z)\in\Omega,t\in[0,T]}
|\ub(\xb,\,z,\,t)|$. Then, from \eqref{eq:27}, we find
\begin{equation*}
    \left|\Phi^1_{z,t}(\ab) - \Phi^0_{z,t}(\ab) \right| \leq UT.
\end{equation*}
Using this bound as the starting point in the iteration \eqref{eq:29},
we derive that
\begin{align*}
  \left|\Phi^{k+1}_{z,t}(\ab) - \Phi^k_{z,t}(\ab) \right|
  &\leq \left(-C \log\epsilon\right)^k UT \dfrac{t^k}{k!} + C\epsilon 
    T \sum_{j=0}^{k-1} \dfrac{(-C\log\epsilon)^j t^j}{j!} \\
  &\leq \left(-C \log\epsilon\right)^k UT \dfrac{t^k}{k!} + C\epsilon 
    T e^{-C\log\epsilon t}. 
\end{align*}
We set $\epsilon = e^{-k}$. With the fact that $0\leq t\leq T$, one
has
\begin{equation*}
  \left|\Phi^{k+1}_{z,t}(\ab) - \Phi^k_{z,t}(\ab) \right|
\leq UT \dfrac{(CT k)^k}{k!} + CT e^{-k(1-CT)}. 
\end{equation*}
For the first term on the right-hand side, one uses the Stirling's
formula $k^k/k! \leq e^k$, and arrives at
\begin{equation*}
  \left|\Phi^{k+1}_{z,t}(\ab) - \Phi^k_{z,t}(\ab) \right|
  \leq UT {(CT e)^k} + CT e^{-k(1-CT)}. 
\end{equation*}
One lets $T$ to be such that $CT\leq e^{-3/2}< 1/2$. One then has
\begin{equation*}
  \left|\Phi^{k+1}_{z,t}(\ab) - \Phi^k_{z,t}(\ab) \right|
  \leq (C+U)T  e^{-\frac{k}{2}}. 
\end{equation*}
This estimate holds for arbitrary $(\ab,z)\in\Omega$ and
$t\in[0,T]$. Thus $\{\Phi_{z,t}^k(\ab)\}_{k\ge 0}$ is a uniformly
Cauchy sequence, and converges to a limit denoted as $\Phi_{z,t}(\ab)$
in \textcolor{red}{$C(\bar\Omega\times[0,\,T])$}. Taking the limit in
\eqref{eq:27}, one proves that this limit is indeed a solution to
\eqref{eq:21} and hence to \eqref{eq:20}.

Given that $T>0$ is independent of $\ab$ and $z$, the above process can
be repeated for the interval $[T,\,2T]$, $2T,\,3T]$, etc, and one thus
obtains a solution $\Phi_{z,t}(\ab)$ for \eqref{eq:21}/\eqref{eq:20}
that is valid for all time.

For uniqueness, we assume that $\tilde\Phi_{z,t}$ is another solution
satisfying \eqref{eq:21},
\begin{equation}
  \label{eq:30}
  \tilde\Phi_{z,t}(\ab) = \ab + \int_0^t 
  \ub(\tilde\Phi_{z,\tau}(\ab),\,z,\,\tau)d\tau,\qquad\forall\,\ab\in\M. 
\end{equation}
Taking the difference between this and \eqref{eq:21}, we have
\begin{equation*}
\Phi_{z,t}(\ab) -   \tilde\Phi_{z,t}(\ab) = \int_0^t 
  \left(\ub(\Phi_{z,\tau}(\ab),\,z,\,\tau) -
    \ub(\tilde\Phi_{z,\tau}(\ab),\,z,\,\tau) \right) d\tau. 
\end{equation*}
Taking absolute values on both sides and using the quasi-Lipschitz
continuity of $\ub$ together with the inequality \eqref{eq:28}, we
derive that
\begin{equation*}
\left|\Phi_{z,t}(\ab) -   \tilde\Phi_{z,t}(\ab)\right|  \leq - C
\log\epsilon \int_0^t 
  \left|\Phi_{z,\tau}(\ab) -
    \tilde\Phi_{z,\tau}(\ab)\right| d\tau + C\epsilon t,
\end{equation*}
for an arbitrary $0<\epsilon < 1$. By the Gronwall inequality
(see Lemma \ref{l:gronwall}), one obtains that
\begin{equation*}
  \left|\Phi_{z,t}(\ab) -   \tilde\Phi_{z,t}(\ab)\right|  \leq
   C \epsilon t e^{-C \log\epsilon t}. 
\end{equation*}
For any fixed $t$, one can pick an $\epsilon$ to make the right-hand
side arbitrarily small. Thus this inequality can only hold if
$\Phi_{z,t}$ and $\tilde\Phi_{z,t}$ are identical, and the
uniqueness is proven. 
\end{proof}

From Theorem \ref{lem:elliptic-reg} and Lemma \ref{l:1} we conclude
that Step 3 is well defined, as long as the sequence $q^n$ of the PV
remains uniformly bounded, which is obviously true thanks to
\eqref{eq:44} in Step 4. The next iteration $q^{n+1}$ of the PV 
is well defined, once
$\Phi^n_{z,t}$ is obtained through Step 3.
We now show that the sequence of triplets
$\{(\Phi_{z,t}^n,\,\ub^n,\,q^n)\}_{n\ge 0}$ thus generated by
iterations of Steps 1-4 converges. Specifically, we have
\begin{lemma}\label{l:2}
  There exists a $T>0$ that depends on $\Omega$ and $|q_0|_\infty$
  only, and a unique triplet of functions $(\Phi_{z,t},\, \ub,\, q)$
  such that, as $n\longrightarrow\infty$,
  \begin{align*}
    \int_\Omega \left| \Phi^n_{z,t}(\ab) - \Phi_{z,t}(\ab) \right| d\ab dz 
    &\longrightarrow 0,\qquad \textrm{
    uniformly over }[0,\,T],\\ 
    \int_\Omega \left|\ub^n(\xb,\,z,\,t) - \ub(\xb,\,z,\,t)\right| d\xb dz 
    &\longrightarrow 0,\qquad\textrm{
    uniformly over }[0,\,T],
  \end{align*}
  and, for any $f\in\mathcal{C}(\M)$, 
  \begin{equation*}
  \int_\Omega f(\xb)\left(q^n(\xb,\,z,\,t) - q(\xb,\,z,\,t)\right)  d\xb dz 
    \longrightarrow 0,\qquad \textrm{
    uniformly over }[0,\,T].
  \end{equation*}
\end{lemma}
\begin{proof}
  We consider the sequence of flow maps $\Phi^n_{z,t}$ first. Taking
  the difference of the equation \eqref{eq:21} for the indices $n$ and
  $n-1$, we have, for $\forall\,\ab\in\M$,
  \begin{align*}
    \Phi^n_{z,t}(\ab) - \Phi^{n-1}_{z,t}(\ab)
    &= \int_0^t 
\left(\ub^n(\Phi^n_{z,\tau}(\ab),\,z,\,\tau) -
  u^{n-1}(\Phi^{n-1}_{z,t}(\ab) , z, \tau) \right) d\tau\\
    &= \int_0^t 
\left(\ub^n(\Phi^n_{z,\tau}(\ab),\,z,\,\tau) -
  u^{n}(\Phi^{n-1}_{z,t}(\ab) , z, \tau) \right) d\tau\\
&\hphantom{=} + \int_0^t 
\left(\ub^n(\Phi^{n-1}_{z,\tau}(\ab),\,z,\,\tau) -
  u^{n-1}(\Phi^{n-1}_{z,t}(\ab) , z, \tau) \right) d\tau
  \end{align*}
  Taking absolute values and then volume averages on both sides, we have
  \begin{multline}
      \label{eq:32}
  \dfrac{1}{|\Omega|} \int_\Omega\left|\Phi^n_{z,t}(\ab) -
    \Phi^{n-1}_{z,t}(\ab)\right|  d\ab dz\\
  {}  \leq {} \dfrac{1}{|\Omega|} \int_0^t  \int_\Omega 
\left|\ub^n(\Phi^n_{z,\tau}(\ab),\,z,\,\tau) -
  u^{n}(\Phi^{n-1}_{z,t}(\ab) , z, \tau) \right| d\ab dz  d\tau+ {} \\
\dfrac{1}{|\Omega|} \int_0^t  \int_\Omega 
\left|\ub^n(\xb,\,z,\,\tau) -
  u^{n-1}(\xb , z, \tau) \right| d\xb dz  d\tau.
  \end{multline}
In the above, a change of variable $\xb = \Phi^{n-1}_{z,t}(\ab)$ and
the fact that $\Phi^{n-1}_{z,t}$ conserves area have been used in the
second integral.

For the first integral in \eqref{eq:32}, one uses the quasi-Lipschitz
continuity of the velocity field $\ub^n$, together with the concavity
of the scalar function $\varphi$ and Jensen's inequality,  to derive that
\begin{multline}
  \label{eq:33}
\dfrac{1}{|\Omega|} \int_0^t  \int_\Omega 
\left|\ub^n(\Phi^n_{z,\tau}(\ab),\,z,\,\tau) -
  u^{n}(\Phi^{n-1}_{z,t}(\ab) , z, \tau) \right| d\ab dz  d\tau \leq \\
\int_0^t \lambda\left( \dfrac{1}{|\Omega|} \int_\Omega 
\left|\Phi^n_{z,\tau}(\ab) -\Phi^{n-1}_{z,t}(\ab) \right| d\ab dz\right)  d\tau.
\end{multline}

For the second integral on the right-hand side of \eqref{eq:32}, we
express the velocity fields $\ub^n$ and $\ub^{n-1}$ using the
Green's function $G^{\xb,z}(\cdot,\,\cdot)$ and the PV $q^n$ and
$q^{n-1}$,
\begin{align*}
&\dfrac{1}{|\Omega|} \int_0^t  \int_\Omega 
\left|\ub^n(\xb,\,z,\,\tau) -
                 u^{n-1}(\xb , z, \tau) \right| d\xb dz  d\tau \\
{}={} &\dfrac{1}{|\Omega|} \int_0^t  \int_\Omega 
\left|\int_\Omega \left( \nabla^\perp_{\xb}  G^{\xb,z}(\xib,\,\eta)
 q^n(\xib,\eta,\tau) - \nabla^\perp_{\xb} G^{\xb,z}(\xib,\,\eta)
 q^{n-1}(\xib,\eta,\tau)\right) d\xib d\eta 
         \right| d\xb dz  d\tau \\
{}={} &\dfrac{1}{|\Omega|} \int_0^t  \int_\Omega 
\left|\int_\Omega \left( \nabla^\perp_{\xb}  G^{\xb,z}(\xib,\,\eta)
 q_0(\Phi^{n-1}_{\eta,-\tau}(\xib),\eta) - \nabla^\perp_{\xb} G^{\xb,z}(\xib,\,\eta)
 q_0(\Phi^{n-2}_{\eta,-\tau}(\xib),\eta)\right) d\xib d\eta 
         \right| d\xb dz  d\tau \\
{}={} &\dfrac{1}{|\Omega|} \int_0^t  \int_\Omega 
\left|\int_\Omega \left( \nabla^\perp_{\xb}
        G^{\xb,z}(\Phi^{n-1}_{\eta,\tau}(\bb),\,\eta)
  - \nabla^\perp_{\xb} G^{\xb,z}(\Phi^{n-2}_{\eta,\tau}(\bb),\,\eta)\right)
 q_0(b,\eta) d\bb d\eta 
         \right| d\xb dz  d\tau \\
{}\leq {} &\dfrac{|q_0|_\infty }{|\Omega|} \int_0^t  \int_\Omega 
\int_\Omega \left| \nabla^\perp_{\xb}
        G^{\xb,z}(\Phi^{n-1}_{\eta,\tau}(\bb),\,\eta)
  - \nabla^\perp_{\xb} G^{\xb,z}(\Phi^{n-2}_{\eta,\tau}(\bb),\,\eta)\right|
         d\xb dz  d\bb d\eta d\tau.
\end{align*}
It can be shown, in exactly the same way, that a bound similar to
\eqref{eq:19} applies to the innermost integral. We then
utilize the Jensen's inequality again to derive that 
\begin{multline}
  \label{eq:34}
\dfrac{1}{|\Omega|} \int_0^t  \int_\Omega 
\left|\ub^n(\xb,\,z,\,\tau) -
                 \ub^{n-1}(\xb , z, \tau) \right| d\xb dz  d\tau 
               \leq\\
C {|q_0|_\infty } \int_0^t \lambda\left( \dfrac{1}{|\Omega|} \int_\Omega 
\left| \Phi^{n-1}_{\eta,\tau}(\bb)  - \Phi^{n-2}_{\eta,\tau}(\bb)\right|
 d\bb d\eta\right) d\tau.  
\end{multline}

Combining \eqref{eq:33} and \eqref{eq:34} in \eqref{eq:32}, and after
denoting
\begin{equation*}
  \delta^n(t) = \dfrac{1}{|\Omega|} \int_\Omega 
\left| \Phi^{n}_{\eta,\tau}(\ab)  - \Phi^{n-1}_{\eta,\tau}(\ab)\right|
 d\ab d\eta, 
\end{equation*}
we have
\begin{equation}
  \label{eq:35}
  \delta^n(t) \leq C \int_0^t \lambda(\delta^n(\tau)) d\tau +
  C |q_0|_\infty \int_0^t \lambda(\delta^{n-1}(\tau)) d\tau,\qquad
  n\ge 1.
\end{equation}
Following Marchioro and Pulvirenti (\cite{Marchioro1994-yt}), we
define 
\begin{equation*}
  \rho^N(t) = \sup_{n\ge N} \delta^n(t). 
\end{equation*}
For a fixed $N$, one considers \eqref{eq:35} for $n\ge N$, and derives that
\begin{equation}
 \label{eq:36}
  \rho^N(t) \leq C(1+|q_0|_\infty) \int_0^t \lambda(\rho^{N-1}(\tau)) d\tau.
\end{equation}
This inequality is in exactly the same form as \eqref{eq:37}, and by
the same technique, one can show that, for a $T$ depending on $\Omega$
and $|q_0|_\infty$ only, 
$\rho^N(t)$ is an exponentially decaying geometric sequence, and
$\Phi^n_{z,t}(\ab)$ is a uniformly Cauchy sequence under the $L^1$-norm over
$[0,\,T]$. Therefore, there exists a unique function $\Phi_{z,t}$ such
that
\begin{equation}\label{eq:38}
  \Phi^n_{z,t} \longrightarrow \Phi_{z,t} \qquad \textrm{in } C([0,T];
  L^1(\Omega)). 
\end{equation}


Next, we show the convergence of $q^n$. We define
\begin{equation*}
  q(\xb,\,z,\,t) = q_0(\Phi_{z,-t}(\xb),\,z). 
\end{equation*}
We now show that $q^n$ converges to $q$ in a certain sense. We let
$f\in\mathcal{C}^1(\bar\M)$. Then, we find that
\begin{align*}
  &\left| \int_0^1\!\int_\M f(\xb)\left( q^n(\xb,\,z,\,t) - q(\xb,\,z,\,t) 
  \right) d\xb dz\right|\\
 {}={} &   \left| \int_0^1\!\int_\M f(\xb)\left( q_0(\Phi^{n-1}_{z,-t}(\xb),\,z) - q_0(\Phi_{z,-t}(\xb),\,z) 
   \right) d\xb dz\right|\\
{}={} &   \left| \int_0^1\!\int_\M \left( f(\Phi^{n-1}_{z,t}(\ab)) -
        f(\Phi_{z,t}(\ab)) \right) q_0(\ab,\,z) d\ab dz\right|\\
{}\leq{} &  |q_0|_\infty \int_0^1\!\int_\M \left| f(\Phi^{n-1}_{z,t}(\ab)) -
           f(\Phi_{z,t}(\ab)) \right| d\ab dz\\
{}\leq{} &  |q_0|_\infty |\nabla f|_\infty \int_0^1\!\int_\M \left| \Phi^{n-1}_{z,t}(\ab) -
           \Phi_{z,t}(\ab) \right| d\ab dz. 
\end{align*}
Thanks to the convergence \eqref{eq:38} of $\Phi^n_{z,t}$, we have
that
\begin{equation*}
  \left| \int_0^1\!\int_\M f(\xb)\left( q^n(\xb,\,z,\,t) - q(\xb,\,z,\,t) 
  \right) d\xb dz\right| \longrightarrow 0,\qquad\textrm{as }
n\rightarrow \infty,\,\textcolor{red}{\textrm{uniformly over }} [0,\,T]. 
\end{equation*}
By a continuity argument, one can show that the above is also true for
general $f\in\mathcal{C}(\bar\M)$. Hence one concludes that $q^n$
converges to $q$ weakly in the sense that, for every
$f\in\mathcal{C}(\bar\M)$,
\begin{equation*}
  \int_0^1\!\int_\M f(\xb)\left( q^n(\xb,\,z,\,t) - q(\xb,\,z,\,t) 
  \right) d\xb dz\longrightarrow 0,\qquad\textrm{as }
n\rightarrow \infty,\,\textcolor{red}{\textrm{uniformly over }} [0,\,T]. 
\end{equation*}

Finally, we show the convergence of $\ub^n$. We let
\begin{equation*}
  \ub(\xb,\,z,\,t) = \int_0^1 \int_\M \nabla^\perp_x
  G^{\xb,z}(\xib,\,\eta) q(\xib,\,\eta,\,t) d\xib d\eta. 
\end{equation*}
Taking the difference of $\ub^n$ and $\ub$, and integrate the absolute
value of the difference over $\Omega$, we have
\begin{multline*}
  \int_0^1 \int_\M | \ub^n(\xb,\,z,\,t) - \ub(\xb,\,z,\,t) | d\xb dz \leq
  {}\\
\int_0^1 \int_\M \left|\int_0^t \int_\M \left(\nabla^\perp_x
    G^{\xb,z}(\xib,\,\eta) q^n(\xib,\,\eta,\,t)- \nabla^\perp_x
    G^{\xb,z}(\xib,\,\eta) q(\xib,\,\eta,\,t)\right) d\xib d\eta \right|
d\xb dz. 
\end{multline*}
Again, using the specifications for $q^n$ and $q$ with the
corresponding flow maps, we have
\begin{align*}
  &\int_0^1 \int_\M | \ub^n(\xb,\,z,\,t) - \ub(\xb,\,z,\,t) | d\xb dz\\
{}\leq{} &|q_0|_\infty \int_0^1 \int_\M \int_0^1\!\int_\M \left|\nabla^\perp_x
    G^{\xb,z}(\Phi^{n-1}_{\eta,t}(\ab),\,\eta)- \nabla^\perp_x
    G^{\xb,z}(\Phi_{\eta,t}(\ab),\,\eta)\right| d\xb dz d\ab d\eta\\
{}\leq{} &C |q_0|_\infty \int_0^1 \int_\M\varphi\left( \left|
           \Phi^{n-1}_{\eta,t}(\ab)- \Phi_{\eta,t}(\ab)\right|\right)  d\ab d\eta\\
{}\leq{} & -C |q_0|_\infty \log\epsilon \int_0^1 \int_\M \left|
           \Phi^{n-1}_{\eta,t}(\ab)- \Phi_{\eta,t}(\ab)\right|  d\ab
           d\eta + C |q_0|_\infty \epsilon |\Omega|. 
\end{align*}
The above holds for arbitrary $\epsilon\in(0,\,1)$. Given that
$\Phi^n_{z,t}$ converges to $\Phi_{z,t}$ uniformly in $t$ in the $L^1$
norm, the same holds for $\ub^n$ and $\ub$ as well. 
\end{proof}

\begin{proof}[Proof of Theorem \ref{t:1}]
Let $\{\Phi^n_{z,t},\,\ub^n,\, q^n\}_{n\ge 0}$ be the sequence of
tripplets generated by an iteration of the iterative procedure. By
Lemma \ref{l:1}, the sequence converges to a tripplet
$\{\Phi_{z,t},\,\ub,\,q\}$ on an interval $[0,\,T_1]$, where $T_1$
depends on $\Omega$ and $|q_0|_\infty$ only. We now verify that 
the limit $(\Phi_{z,t},\, \ub,\, q)$ indeed
solves the QG system, one starts with the equation for $\Phi^n_{z,t}$
and $\ub^n$,
\begin{equation*}
  \Phi^n_{z,t}(\ab) = \ab + \int_0^t \ub^n (\Phi^n_{z,t}(\ab),\,z,\,t) d\tau.
\end{equation*}
We take the limit on both sides. For an arbitrary $t$, the left-hand
side converges to $\Phi_{z,t}(\ab)$ in the $L^1$-norm, and the
right-hand side converges to $\ab + \int_0^t
\ub(\Phi_{z,\tau}(\ab),\,z,\,\tau) d\tau$. Thus, in the limit, one
has, for an arbitrary $t$,
\begin{equation*}
  \Phi_{z,t}(\ab) = \ab + \int_0^t \ub(\Phi_{z,t}(\ab),\,z,\,t)
  d\tau,\qquad \textrm{in } L^1(\Omega). 
\end{equation*}
Thanks to the quasi-Lipschitz continuity of $\ub$, $\Phi_{z,t}(\ab)$
is continuous in $\Omega$, and
the above
holds for all $(\ab,\,z)\in\Omega$. Thus, 
the limit functions $(\Phi_{z,t},\,\ub,\, q)$ satisfy the
equations \eqref{eq:13}, \eqref{eq:2}, \eqref{eq:21}, and
\eqref{eq:26} over $[0,\,T]$. They are a weak solution to the QG
system. 

To show that $(\Phi_{z,t},\,\ub,\,q)$ is the only
solution. We assume that $(\tilde\Phi_{z,t},\,\tilde\ub,\, \tilde q)$
is another set of solutions, satisfying \eqref{eq:21} and
\eqref{eq:26} for the same initial data $q_0$. Taking the difference
of the equation \eqref{eq:21} for $\Phi_{z,t}$ and $\tilde
\Phi_{z,t}$,  we have
\begin{equation*}
  \Phi_{z,t}(\ab) - \tilde\Phi_{z,t}(\ab) = \int_0^t
  \left(\ub(\Phi_{z,\tau}(\ab),\,z,\,\tau) -
    \tilde\ub(\tilde\Phi_{z,\tau}(\ab),\, z,\, \tau)\right) d\tau. 
\end{equation*}
As before, we split the right-hand side into two parts, take the
absolute value on both sides, and then take the volume average to
obtain
\begin{multline}
  \label{eq:39}
  \dfrac{1}{|\Omega|} \int_\Omega \left| \Phi_{z,t}(\ab) -
    \tilde\Phi_{z,t}(\ab)\right| d\ab dz \leq {}\\
\dfrac{1}{|\Omega|}\int_0^t \int_\Omega \left|
  \ub(\Phi_{z,\tau}(\ab),\, z,\, \tau) -
  \ub(\tilde\Phi_{z,\tau}(\ab),\,z,\,\tau) \right| d\ab dz d\tau 
+ {}\\
\dfrac{1}{|\Omega|}\int_0^t \int_\Omega \left|
  \ub(\tilde\Phi_{z,\tau}(\ab),\, z,\, \tau) -
  \tilde \ub(\tilde\Phi_{z,\tau}(\ab),\,z,\,\tau) \right| d\ab dz d\tau.
\end{multline}
Using the quasi-Lipschitz continuity of $\ub$ on the first term on the
right-hand side, we have
\begin{multline}
  \label{eq:40}
\dfrac{1}{|\Omega|}\int_0^t \int_\Omega \left|
  \ub(\Phi_{z,\tau}(\ab),\, z,\, \tau) -
  \ub(\tilde\Phi_{z,\tau}(\ab),\,z,\,\tau) \right| d\ab dz d\tau \leq
{}\\
\int_0^t \varphi\left( \dfrac{1}{|\Omega|} \int_\Omega \left|
  \Phi_{z,\tau}(\ab) -
  \tilde\Phi_{z,\tau}(\ab) \right| d\ab dz \right) d\tau.
\end{multline}
For the second term, we express the velocity fields using the Green's
function $G$ on the corresponding potential vorticity fields, and
after a series of changes of variables and an application of the
quasi-Lipschitz continuity of the gradient of $G$, we obtain
\begin{multline}
  \label{eq:41}
\dfrac{1}{|\Omega|}\int_0^t \int_\Omega \left|
  \ub(\tilde\Phi_{z,\tau}(\ab),\, z,\, \tau) -
  \tilde \ub(\tilde\Phi_{z,\tau}(\ab),\,z,\,\tau) \right| d\ab dz
d\tau \leq {} \\
C |q_0|_\infty \int_0^t \varphi\left( \dfrac{1}{|\Omega|} \int_\Omega \left|
  \Phi_{z,\tau}(\ab) - \tilde \Phi_{z,\tau}(\ab)\right| d\ab dz\right) d\tau.
\end{multline}
Combining \eqref{eq:40} and \eqref{eq:41} in \eqref{eq:39}, and after
setting
\begin{equation*}
  \delta(t) = \dfrac{1}{|\Omega|} \int_\Omega \left|
  \Phi_{z,\tau}(\ab) - \tilde \Phi_{z,\tau}(\ab)\right| d\ab dz,
\end{equation*}
we arrive at the inequality
\begin{equation*}
  \delta(t) \leq C(1+|q_0|_\infty) \int_0^t \lambda(\delta(\tau))
  d\tau. 
\end{equation*}
Using the inequality \eqref{eq:28}, with $0<\epsilon < 1$ arbitrary,
for the function $\varphi$ in the above, we derive that
\begin{equation*}
  \delta(t) \leq C(1+|q_0|_\infty) \left( -\log\epsilon \int_0^t
    \delta(\tau)  d\tau + \epsilon t\right). 
\end{equation*}
By Gronwall's inequality, we have
\begin{align*}
  \delta(t) \leq {}& C(1+|q_0|_\infty ) \epsilon t
  e^{-C(1+|q_0|_\infty) \log \epsilon t}\\
  {}={} & C(1+|q_0|_\infty ) t \epsilon^{1-C(1+|q_0|_\infty) t}.
\end{align*}
On an interval $[0,\,T_2]$ where $C(1+|q_0|_\infty) T_2\leq \frac{1}{2}$,
we have
\begin{equation*}
  \delta(t) \leq C(1+|q_0|_\infty ) T_2 \epsilon^{1-C(1+|q_0|_\infty) T_2}
  \leq \dfrac{1}{2} \epsilon^{\frac{1}{2}}. 
\end{equation*}
Given the arbitrariness of $0<\epsilon <1$, the above can only be true
when
\begin{equation*}
  \delta(t) \equiv 0,\qquad \forall\,t\in[0,\,T_2]. 
\end{equation*}
Thus $\Phi_{z,t}$ must be unique, and as a
consequence, $q(\xb,\,z,\,t)$ and $\ub(\xb,\,z,\,t)$ must also be
unique. The tripplet $\{\Phi_{z,t},\,\ub,\,q\}$ is a unique solution
to the generalized QG system on the interval $[0,\,T]$, where $T =
\min\{T_1,\,T_2\}$ depends on $\Omega$ and $|q_0|_\infty$ only. 

By repeating the argument above we extend the conclusion
for all time. 
The proof of Theorem \ref{t:1} is complete.
  
\end{proof}

\section{A classical solution}\label{s:classical}

\begin{theorem}\label{t:cls}
 Suppose the initial state $q_0\in
 \mathcal{C}^1(\bar\Omega)$. Then the solution
 $q\in\mathcal{C}^1(\bar\Omega\times\mathbb{R})$, and the QG
 system holds in the classical sense. 
\end{theorem}

\begin{proof}
From the assumption it is clear that the initial state $q_0$ is
uniformly H\"older continuous over $\Omega$, i.e.~there exist
constants $C>0$ and $0<\alpha\leq 1$ such that
\begin{equation*}
  \left| q_0(\xb_1,z_1) - q_0(\xb_2,z_2)\right| \leq C\left( |\xb_1 - \xb_2|
      + |z_1-z_2| \right)^\alpha,
    \qquad\forall\,(\xb_1,z_1),\,(\xb_2,z_2)\in\bar\Omega. 
  \end{equation*}

With this H\"older continuity on $q_0$, we now show that, for all time
$t\in[0,\,T]$, the map $(\ab,z)\longrightarrow (\Phi_{z,t}(\ab), z)$ is
uniformly H\"older continuous. We 
let $(\ab_1,z_1),\,(\ab_2,z_2)\in\bar\Omega$ be arbitrary. Using the
relation \eqref{eq:21}, and  the
quasi-Lipschitz continuity of the velocity field $\ub$, we derive that
\begin{align*}
  &\left| \Phi_{z_1,t}(\ab_1) - \Phi_{z_2,t}(\ab_2)\right| + |z_1 -
    z_2| \\ {}\leq
  &|\ab_1 - \ab_2| + |z_1 - z_2| +  \int_0^t \left| \ub(\Phi_{z_1,\tau}(\ab_1),z_1,\tau) -
    u(\Phi_{z_2,\tau}(\ab_2),z_2,\tau)\right| d\tau\\ {}\leq 
  &|\ab_1 - \ab_2| +|z_1 - z_2| +  C \int_0^t\lambda\left( \left|\Phi_{z_1,\tau}(\ab_1) -
    \Phi_{z_2,\tau}(\ab_2)\right| + |z_1 - z_2|\right)  d\tau,
\end{align*}
where $\lambda$ is a monotonic scalar function defined in Section
\ref{s:elliptic}. 
It can be shown (see  Proposition A.1, Appendix
2.1, of \cite{Marchioro1994-yt})  that
\begin{equation*}
\left| \Phi_{z_1,t}(\ab_1) - \Phi_{z_2,t}(\ab_2)\right| + |z_1 -
    z_2| \leq\left\{
    \begin{aligned}
&\left(|\ab_1 - \ab_2| + |z_1 -
  z_2|\right)^{e^{-Ct}} \cdot e^{1-e^{-Ct}}, & & t\leq t_0,\\
&1 + C(t-t_0), & & t> t_0,
    \end{aligned}\right.
\end{equation*}
where $t_0$ is the time when the first expression on the right reaches
value $1$, and the constant $C$ is inherited directly from the estimate
\eqref{eq:19}. It is easy to verify that $t_0$ increases as the
initial distance $|\ab_1 - \ab_2|+|z_1 - z_2|$ decreases. 
Thus, there exists $0<\alpha'\equiv \alpha'(T)\leq 1$, and $C\equiv
C(T,\Omega)>0$, such that  
  \begin{equation*}
\left| \Phi_{z_1,t}(\ab_1) - \Phi_{z_2,t}(\ab_2)\right| + |z_1 -
    z_2| \leq C\left( |\ab_1 - \ab_2| + |z_1 -
      z_2|\right)^{\alpha'}.   
  \end{equation*}
By reversing the direction of the flow, it can be shown that the same
estimate holds for the inverse map $\Phi_{z,t}^{-1}$, i.e., for
$(\xb_1,z_1),\,(\xb_2,z_2)\in\Omega$, $t\in[0,\,T]$ arbitrary,
  \begin{equation*}
\left| \Phi^{-1}_{z_1,t}(\xb_1) - \Phi^{-1}_{z_2,t}(\xb_2)\right| + |z_1 -
    z_2| \leq C\left( |\xb_1 - \xb_2| + |z_1 -
      z_2|\right)^{\alpha'}.   
  \end{equation*}

With the H\"older continuity of the flow map, we now show that the PV
$q$ is also H\"older continuous for all time $t\in[0,\,T]$. Again, we let
$(\xb_1,z_1),\,(\xb_2,z_2)\in\bar\Omega$ and $t\in[0,\,T]$  be arbitrary.
\begin{align*}
  &\left| q(\xb_1,z_1, t) - q(\xb_2,z_2, t)\right| \\
{}=& \left| q_0\left(\Phi_{z_1,t}^{-1}(\xb_1),z_1\right)-
     q_0\left(\Phi_{z_2,t}^{-1}(\xb_2),z_2\right) \right| \\
{}\leq& C \left( \left|
  \Phi_{z_1,t}^{-1}(\xb_1)-\Phi_{z_2,t}^{-1}(\xb_2)\right| + |z_1 -
  z_2|\right)^\alpha\\
{}\leq& C \left( \left|
  \xb_1 - \xb_2\right| + |z_1 -
  z_2|\right)^{\alpha\alpha'}.
\end{align*}

Thus, if $q_0$ is H\"older continuous over $\bar\Omega$, then
$q(\cdot,\cdot,t)$ is also H\"older continuous for all time $t\in[0,\,T]$. 
Given the H\"older continuity of $q$, a classical result in elliptic
PDE theories (\cite{Courant1989-gn}) states that the streamfunction,
as the solution to the 
elliptic BVP \eqref{eq:2}, is $\mathcal{C}^2$, and the velocity field
$\ub$, as the gradient of the streamfunction, is then $\mathcal{C}^1$
continuous. It then follows that $\Phi_{z,t}$, as a solution to the
IVP \eqref{eq:20}, is also differentiable with  respect to the initial
data as well as to the parameter $z$.

Finally, given that $q_0 \in C^1(\bar\Omega)$, the PV solution
$q(\xb,z,t)$ is differential in each of its variables. Indeed, its
time derivative can be calculated as follows,
\begin{align*}
\dfrac{\p}{\p t} q(\xb,z,t) &= \dfrac{\p}{\p t}
  q_0\left(\Phi_{z,t}^{-1}(\xb),\,z\right) = \dfrac{\p}{\p \ab}
  q_0\left(\Phi_{z,t}^{-1}(\xb),z\right) \cdot \dfrac{\p
                              \Phi_{z,t}^{-1}(\xb) }{\p t}  \\
  {}&= -\dfrac{\p}{\p \ab}
  q_0\left(\Phi_{z,t}^{-1}(\xb),z\right) \cdot \left(\dfrac{\p \Phi_{z,t}(\ab) }{\p
      \ab }\right)^{-1}\ub(\xb,t). 
\end{align*}
Similarly,
\begin{align*}
  &\dfrac{\p}{\p\xb} q(\xb,z,t) = \dfrac{\p}{\p\ab}
  q_0\left(\Phi_{z,t}^{-1}(\xb),z\right) \left(\dfrac{\p \Phi_{z,t}(\ab) }{\p
      \ab }\right)^{-1} ,\\
 & \dfrac{\p}{\p z} q(\xb,z,t) =
  \dfrac{\p}{\p\ab}q_0\left(\Phi_{z,t}^{-1}(\xb),z\right) \cdot \dfrac{\p
  \Phi_{z,t}^{-1}(\xb)}{z} + \dfrac{\p}{\p z} q_0\left(\Phi_{z,t}^{-1}(\xb),z\right).
\end{align*}
The above calculations are legitimate, because each term has been
shown to be well defined. Thus, the PV solution $q(\xb,z,t)$ is
differentiable against each of its variables. It is also apparent that
\begin{equation*}
  \dfrac{\p }{\p t} q + \ub\cdot\nabla q = 0,
\end{equation*}
that is, the QG equation holds in the classical sense. This completes
the proof that $q$ is a classical solution to the system. 
\end{proof}

\section{Remarks}
In this work, the 3D quasi-geostrophic equation over a cylindrical domain
with multiply connected cross-section is studied. The model is
equipped with no-flux boundary conditions on the lateral, homogeneous
Neumann on the top and bottom boundary interfaces, constant
circulation along each interior boundary loop. The QG flow behaves more
like a two-dimensional flow despite being in a three-dimensional
domain. 
With the setup proposed here, the classical strategies of Yudovich and
Kato (\cite{Yudovich1963-bj,Kato1967-pe}; see also
\cite{Marchioro1994-yt}) can be applied to the present case without
much 
difficulty, and the global existence and uniqueness of a weak solution
is proven, given that the initial PV field is essentially bounded. If
the initial PV field is more regular, in $C^1$ for example, then the
solution is shown to satisfy the PDE in the classical sense. 

A natural question to ask is whether the same classical technique can
be applied to the case where a transport equation is posed on the top
and/or bottom boundary interface, in place of the homogeneous Neumann
boundary conditions, as in \cite{Desjardins1998-hb,Novack2020-gi,
  Novack2020-vv}. This 
change will make the model more broadly applicable in geophysics. 
However, this change will also immediately bring about a few technical
difficulties. On the one hand, with non-homogeneous
conditions/constraints on the boundaries, various compatibility
conditions will have to be imposed to ensure the viability and
regularity of any possible solutions; see a brief discussion on this
matter at the end of section 2, and discussions in
\cite{Novack2020-gi, Novack2020-vv}.  
On the other hand, with the transport equation replacing the homogeneous
Neumann boundary interfaces, the model becomes much more complex, with
the infamous SQG-like dynamics on the top and/or bottom boundary
interface (\cite{Constantin1994-me, Constantin1994-gm}), and QG-like
dynamics in the interior. These fascinating 
issues will be investigated elsewhere.

\appendix
\section{Estimates on the Green's function to the non-standard
  elliptic BVP}\label{sec:green}

It is well known that the Green's function to the elliptic boundary
value problem over a domain with smooth boundaries enjoy the
power-law estimates (\cite{Superiore1963-wl, Gr-uter1982-fy}). 
The cylindrical domain for our non-standard elliptic boundary value
problem \eqref{eq:2}  possesses an intrinsic geometric
non-smoothness, that is, the sharp edges between the top/bottom surface
and the lateral side. However, with Neumann boundary conditions on the
top and bottom surface, this singularity on the boundary can be dealt
with by an expansion of the domain in the vertical direction onto
$[ -1,\,1]$, with periodic boundary conditions along this interval
(\cite{Courant1953-gq}).  
We denote this expanded domain by $\tilde\Omega =
\M\times (-1,\,1)$, 
and, for each $(\xb,\,z)\in\tilde\Omega$, we let $\Phi^{\xb,z}(\xib,\eta)$
be the fundamental solution of the Laplace operator 
\begin{equation*}
  \Delta_3 = \Delta + \dfrac{\p}{\p z}\left(\dfrac{\p}{\p
      z}(\cdot)\right)
\end{equation*}
centered at $(\xb,\,z)$. To construct the Green's function, we
 let $\phi^{\xb,z}(\xib,\eta)$ be the corresponding corrector function, satisfying
\begin{equation}
  \label{eq:10}
  \left\{
    \begin{aligned}
      \Delta_3 \phi^{\xb,z}(\xib,\,\eta) &= \frac{1}{|\tilde\Omega|},\\
      & &\\
    \left.\left(\phi^{\xb,z}(\xib,\eta) - \Phi^{\xb,z}(\xib,\eta)\right)\right|_{\Gamma_l}
    &= \textrm{indep.~of }\xib, & & 0\leq l\leq L,\\
      \varphi^{\xb,z}(\xib,\,-1)- \Phi^{\xb,z}(\xib,\,-1) &=
     \varphi^{\xb,z}(\xib,\,1) - \Phi^{\xb,z}(\xib,\,1), & &\xib\in\M, \\ 
      \dfrac{\p \varphi^{\xb,z}}{\p \eta}(\xib,\,-1) -  \dfrac{\p
        \Phi^{\xb,z}}{\p \eta}(\xib,\,-1) &=
     \dfrac{\p\varphi^{\xb,z}}{\p \eta}(\xib,\,1) -
     \dfrac{\p\Phi^{\xb,z}}{\p \eta}(\xib,\,1), & &\xib\in\M,\\
    \int_{\Gamma_l} \dfrac{\p\phi^{\xb,z}}{\p\nb}(\xib,\eta) ds &=
    \int_{\Gamma_l} \dfrac{\p\Phi^{\xb,z}}{\p\nb}(\xib,\, \eta) ds, & &
    0\leq l\leq L,\\
    \int_0^1\!\int_{\M} \phi^{\xb,z}(\xib,\eta) d\xib d\eta&=
    \int_0^1\!\int_\M \Phi^{\xb,z}(\xib,\, \eta) d\xib d\eta. & &
    \end{aligned}\right.
\end{equation}

We let
\begin{equation}
  \label{eq:16}
  G_{\mathrm{per}}^{\xb,z}(\xib, \eta)= \Phi^{\xb,z} (\xib, \eta) - \phi^{\xb,z}(\xib,\eta).
\end{equation}
Then, formally, $G_{\mathrm{per}}^{\xb,z}(\xib, \eta)$ satisfies
\begin{equation}
  \label{eq:24}
  \left\{
    \begin{aligned}
      \Delta G_{\mathrm{per}}^{\xb,z}+ \dfrac{\p^2}{\p
        z^2}G_{\mathrm{per}}^{\xb,z} &= \delta^{\xb,z}(\xib,\eta) -
      \frac{1}{|\tilde\Omega|}, &  
      &(\xib,\,\eta)\in\M\times(-1,\,1),\\
      G_{\mathrm{per}}^{\xb,z}(\xib,\eta)|_{\Gamma_l} &= \textrm{indep.~of }\xib & &\xib\in\p\M,\,\,\eta\in(0,\,1),\\
      G_{\mathrm{per}}^{\xb,z}(\xib,\,-1) &=
      G_{\mathrm{per}}^{\xb,z}(\xib,\,1), & &\xib\in\M, \\ 
      \dfrac{\p G_{\mathrm{per}}^{\xb,z}}{\p \eta}(\xib,\,-1) &= \dfrac{\p
        G_{\mathrm{per}}^{\xb,z}}{\p \eta}(\xib,\,1), & &\xib\in\M,\\
  \int_{\Gamma_l} \dfrac{\p G_{\mathrm{per}}^{\xb,z}}{\p{\bs n}} ds &= 0, &  &0\leq l\leq L, \\
  \int_0^1\!\int_\M G_{\mathrm{per}}^{\xb,z}  d\xib d\eta&= 0, & &\forall\,\eta \in (0,\,1). 
    \end{aligned}\right.
\end{equation}

From $\Gper$, the Green's function $G^{\xb,z}$ for the non-standard
boundary value problem \eqref{eq:2} can be constructed.  For each
$(\xb,z)\in \Omega \equiv \M\times(0,\,1)$,  we let
\begin{equation}
  \label{eq:25}
  G^{\xb,z}(\xib,\eta) = \Gper(\xib,\eta) + G^{\xb,z}_{\textrm{per}}(\xib,-\eta), \qquad
  \forall\,\,(\xib,\eta)\in\M\times(0,\,1). 
\end{equation}
Formally, $G^{\xb,z}(\xib,\eta)$ satisfies 
\begin{equation}
  \label{eq:23}
  \left\{
    \begin{aligned}
      \Delta G^{\xb,z} + \dfrac{\p^2}{\p z^2}G^{\xb,z} &=
      \delta^{\xb,z}(\xib,\eta) - \frac{1}{|\Omega|}, &
      &(\xib,\,\eta)\in\M\times(0,1),\\
      G^{\xb,z}(\xib,\eta)|_{\Gamma_l} &= \textrm{indep.~of }\xib & &\xib\in\p\M,\,\,\eta\in(0,\,1),\\
   \int_{\Gamma_l}
   \dfrac{\p G^{\xb,z}}{\p{\bs n}} ds &= 0, &  &1\leq l\leq L, \\
      \dfrac{\p G^{\xb,z}}{\p \eta} &= 0, & &\xib\in\M,\,\,\eta=0,\,1,\\
  \int_\M G^{\xb,z}  d\xb &= 0. & &\forall\,\eta \in (0,\,1). 
    \end{aligned}\right.
\end{equation}
Here, $\delta^{\xb,z}(\xib,\eta)$ of course is the Dirac Delta function
centered at $(\xb,\,z)$.
We note that $G^{\xb,z}(\xib,\eta)$ obviously satisfies the PDE (1st
equation of \eqref{eq:23}), the  lateral boundary
conditions (2nd and 3rd equations), and the constraint (5th
equation). With a little bit of calculus, it can be shown that it 
also satisfies the homogeneous Neumann boundary conditions on the top
and bottom boundaries (4th equation).

The Green's function $G^{\xb,z}(\xib,\eta)$ to the non-standard elliptic
boundary value 
problem \eqref{eq:2} enjoys the following power function estimates as
$\Gper(\xib,\eta)$ does.

\begin{lemma}\label{lem1}
  The Greens's function for the three-dimensional elliptic boundary
  value problem \eqref{eq:2} over
  the cylinder $\Omega$  satisfies the 
  following estimates,
  \begin{subequations}\label{green}
  \begin{align}
|G^{\xb,z}(\xib,\eta)| &\leq \dfrac{C}{r},\label{eq:g0}\\ 
  |\nabla G^{\xb,z}(\xib,\eta)|  &\leq  \dfrac{C}{r^2},\label{eq:g1}\\
  |\nabla^2 G^{\xb,z}(\xib,\eta)|  &\leq  \dfrac{C}{r^3},\label{eq:g2}
\end{align}
\end{subequations}
where $r = \sqrt{|\xb-\xib|^2 + (z-\eta)^2}$, and each general constant
$C$ in the above depends on the domain $\Omega$ only. 
\end{lemma}


\section{H\"older and quasi-Lipschitz continuity}\label{sec:lipschitz}
In this section, we prove Theorem \ref{lem:elliptic-reg} concerning the
quasi-Lipschitz and H\"older continuity 
of the gradient of the streamfunction $\psi$, and hence of the
horizontal velocity field. This is a direct consequence of the
following lemma.

\begin{lemma}\label{lem2}
  The gradient of the integral of the Green's function
  $G^{\xb,z}(\xib,\eta)$ to the nonstandard 
  boundary value problem \eqref{eq:2} satisfies the quasi-Lipschitz condition, 
  \begin{equation}
    \label{eq:19}
    \int_0^1\!\int_\M \left|D_{\xb,z}  G^{\xb_1,z_1}(\xib,\,\eta) - D_{\xb,z} 
      G^{\xb_2,z_2}(\xib,\,\eta) \right| d\xib d\eta \leq C\lambda(d), 
  \end{equation}
  where
  \begin{equation*}
 d =   \sqrt{|\xb_1 - \xb_2|^2 + (z_1 - z_2)^2},    
  \end{equation*}
  and the scalar function $\lambda(d)$ is defined in the statement of
  Theorem \ref{lem:elliptic-reg}.
\end{lemma}

\begin{proof}
We denote by 
\begin{equation*}
  B(\xb,\,z,\,d) = \left\{ (\xib,\eta)\,|\, \sqrt{| \xib-x|^2 + (\eta-z)^2} \leq d\right\}
\end{equation*}
a ball with radius $d$ and center $(\xb, z)$, and
\begin{equation*}
  B(\xb,\,z,\,d,\,R) = \left\{ (\xib,\eta)\,|\, d\leq \sqrt{| \xib-x|^2 + (\eta-z)^2} \leq R\right\}
\end{equation*}
a shell with inner radius $d$, outer radius $R$, and center
$(\xb,z)$. 
Then
\begin{multline*}
  \int_0^1\!\int_\M \left|D_{\xb,z}  G^{\xb_1,z_1}(\xib,\,\eta) - D_{\xb,z} 
      G^{\xb_2,z_2}(\xib,\,\eta) \right| d\xib d\eta =\\
  \int_{\Omega\cap B(\xb_1,z_1,2d) }
  \left|D_{\xb,z} G^{\xb_1,z_1}(\xib,\,\eta)- D_{\xb,z} G^{\xb_2,z_2}(\xib,\,\eta)\right| d\xib d\eta +{}\\
 \int_{\Omega\backslash B(\xb_1,z_1,2d) }
  \left|D_{\xb,z} G^{\xb_1,z_1}(\xib,\,\eta) - D_{\xb,z} G^{\xb_2,z_2}(\xib,\,\eta)\right| d\xib d\eta.
\end{multline*}
On the first integral on the right-hand side,
\begin{align*}
  &\int_{\Omega\cap B(\xb_1,z_1,2d) }
  \left|D_{\xb,z} G^{\xb_1,z_1}(\xib,\,\eta) - D_{\xb,z} G^{\xb_2,z_2}(\xib,\,\eta)\right| d\xib d\eta \\
  \leq & \int_{B(\xb_1,z_1,2d)} |\nabla G^{\xb_1,z_1}(\xib,\,\eta) | d\xib
         d\eta+  \int_{B(\xb_2,z_2,3d)} |\nabla G^{\xb_2,z_2}(\xib,\,\eta) | d\xib
         d\eta\\
  \leq & \int_{B(\xb_1,z_1,2d)}  \dfrac{C}{r^2} d\xib
         d\eta+  \int_{B(\xb_2,z_2,3d)} \dfrac{C}{r^2}  d\xib
         d\eta\\
  \leq& 2C \omega_3 d + 3 C\Omega_3 d = 5C \omega_3 d.
\end{align*}
Here, $\omega_3$ represents the surface area of the three-dimensional
unit ball. The second integral can be handled with the help of the
Mean Value Theorem and the estimate \eqref{eq:g2},
\begin{multline*}
  \int_{\Omega\backslash B(\xb_1,z_1,2d) }
  \left|D_{\xb,z} G^{\xb_1,z_1}(\xib,\,\eta) - D_{\xb,z} G^{\xb_2,z_2}(\xib,\,\eta)\right| d\xib d\eta\leq \\ C d \int_{\Omega\backslash B(\xb_1,z_1,2d) }
  \left|D^2_{\xb,z} G^{\bar\xb,\bar z}(\xib,\,\eta))\right| d\xib d\eta
  \leq  C d \int_{\Omega\backslash B(\xb_1,z_1,2d) }
  \dfrac{1}{\bar d^3} d\xib d\eta,
\end{multline*}
where $(\bar\xb, \bar z)$ is a certain point on the line segment
between $(\xb_1, z_1)$ and $(\xb_2, z_2)$, $\bar d = \sqrt{|\xib-\bar\xb|^2
  + (\eta - \bar z)^2}$.
We let $d_1$ be the distance between $(\xib,\eta)$ and $(\xb_1,z_1)$, and,
using the fact that $\bar d / d_1 \ge 1/2$, we derive that
\begin{multline*}
  \int_{\Omega\backslash B(\xb_1,z_1,2d) }
  \left|D_{\xb,z} G^{\xb_1,z_1}(\xib,\,\eta) - D_{\xb,z} G^{\xb_2,z_2}(\xib,\,\eta)\right| d\xib d\eta\leq \\ 
  \leq  8C d \int_{\Omega\backslash B(\xb_1,z_1,2d) }
  \dfrac{1}{d^3_1} d\xib d\eta\leq 8C d \int_{ B(\xb_1,z_1,2d, R) }
  \dfrac{1}{ d^3_1} d\xib d\eta \leq 8C\omega_3 d (\log R - \log 2d),
\end{multline*}
where $R$ is chosen to be the diameter of the region $\Omega$. 
Combining the estimates for these two integrals together, we have
\begin{equation*}
  \int_0^1\!\int_\M \left|D_{\xb,z}  G^{\xb_1,z_1}(\xib,\,\eta) - D_{\xb,z} 
      G^{\xb_2,z_2}(\xib,\,\eta) \right| d\xib d\eta
    \leq C\omega_3 d (5 + 8 \log R - 8 \log 2d).  
  \end{equation*}
The claim of the quasi-Lipschitz continuity is proven.

\end{proof}

\begin{proof}[Proof of Theorem \ref{lem:elliptic-reg}]
The H\"older and quasi-Lipschitz estimates of the gradient of $\psi$ can be
established via direct calculation using the formula \eqref{eq:17}. 
We frist verify
that $\psi\in C^1(\M)$. We 
formally differentiate the formula,
\begin{equation}
  \label{eq:214a}
    \nabla\psi(\xb,z) = \int_\Omega D_{\xb,z}
    G^{\xb,z}(\xib,\,\eta)q(\xib,\eta)d\xib d\eta.
\end{equation}
We need to show that the right-hand side is
well-defined. Using the fact that $q$ is essentially bounded, we find
that 
\begin{align*}
  |\nabla \psi(\xb,z)| &\leq \int_\Omega \left|D_{\xb,z} G^{\xb,z}(\xib,\,\eta) \right|\cdot
           |q(\xib,\zeta)| d\xib d\eta\\
  &\leq |q|_\infty \int_\Omega \left| D_{\xb,z} G^{\xb,z}(\xib,\,\eta)\right|d\xib d\eta.
\end{align*}
Using the estimate \eqref{eq:g1} from the above, we proceed,
\begin{align*}
 |\nabla\psi(\xb,z)|  &\leq  C|q|_\infty \int_\Omega
    \dfrac{1}{r^2} d\xib d\eta\\
  &\leq C|q|_\infty  \left\{
    \int_{\substack{(\xib,\eta)\notin B_{\xb,z}\\(\xib,\eta)\in \Omega}}
  \dfrac{1}{r^2}d\xib d\eta  +
  \int_{\substack{(\xib,\eta)\in B_{\xb,z}\\(\xib,\eta)\in \Omega}}
  \dfrac{1}{r^2} d\xib d\eta \right\}\\
  & \leq  C|q|_\infty (|\Omega| + \omega_3).
\end{align*}
Here, $B_{\xb,z}$ represents a unit ball centered at $(\xb,\,z)$,
$|\Omega|$ the volume of the three-dimensional domain 
$\Omega$, and $\omega_3$ the surface area of the unit ball in the
three dimensional space. 
Hence, the right-hand side of \eqref{eq:214a} is well-defined, and the
relation \eqref{eq:214a} holds. 

Next, we show that the first derivative of $\psi$ is quasi-Lipschitz
continuous. We let $(\xb_1,\,z_1)$ 
and $(\xb_2,\,z_2)$ be two arbitrary points in $\Omega$. Then, using the
relation 
\eqref{eq:214a}, we derive that 
\begin{equation}
  \label{eq:214c}
  \left| \nabla \psi(\xb_1,z_1) - \nabla \psi(\xb_2,\,z_2)
  \right| \leq |q|_\infty \int_\Omega
  \left|D_{\xb,z} G^{\xb_1,z_1}(\xib,\,\eta) - D_{\xb,z} G^{\xb_2, z_2}(\xib,\,\eta)\right| d\xib d\eta.
\end{equation}
The estimate \eqref{eq:214j} then follows from an application of Lemma
\ref{lem2} to the above.

The H\"older continuity is a consequence of the quasi-Lipschitz
continuity. Indeed, for any $0 < \lambda < 1$, 
\begin{equation*}
  \dfrac{\left| \nabla\psi(\xb_1,z_1) - \nabla\psi(\xb_2,z_2)\right|}{d^\lambda} \leq C(\Omega,
  |q|_\infty) (1-\ln d) d^{1-\lambda}.
\end{equation*}
For $0<d<1$, the expression $|\ln d|\cdot d^{1-\lambda}$
has a maximum value of $e^{-1}/(1-\lambda)$. Hence, we have that 
\begin{equation*}
  \dfrac{\left| \nabla\psi(\xb_1,z_1) - \nabla\psi(\xb_2,z_2)\right|}{d^\lambda}   \leq C(\Omega,
  |q|_\infty) \dfrac{1}{1-\lambda}.
\end{equation*}
\end{proof}

\section{Gronwall inequality}
\begin{lemma}\label{l:gronwall}
  If $\beta\ge 0$ and $u$ satisfies the integral inequality
  \begin{equation*}
    u(t) \leq \alpha(t) + \int_a^t \beta(s)u(s) ds, \qquad\forall\, t\ge a,
  \end{equation*}
  then
  \begin{equation*}
    u(t) \leq \alpha(t) + \int_a^t \alpha(s) \beta(s)
    \exp\left(\int_s^t \beta(r) dr  \right) ds, \qquad \forall t\ge a. 
  \end{equation*}
  If, in addition, the function $\alpha$ is non-decreasing, then
  \begin{equation*}
    u(t) \leq \alpha(t) \exp\left(\int_a^t \beta(s) ds
    \right),\qquad \forall\, t\ge a. 
  \end{equation*}
\end{lemma}
\bibliographystyle{plain}
\bibliography{references}
\end{document}